\documentclass[12pt]{article}
\usepackage{amsmath,amssymb,amsthm}
\usepackage{graphicx}
\usepackage[colorlinks,
            linkcolor=red,
            anchorcolor=blue,
            citecolor=green
            ]{hyperref}

\newcommand{\C}{{\mathbb C}}
\newcommand{\N}{{\mathbb N}}

\newcommand{\R}{{\mathbb R}}

\newtheorem{lemma}{Lemma}[section]
\newtheorem{proposition}{Proposition}[section]
\newtheorem{theorem}{Theorem}[section]

\newtheorem{remark}{Remark}[section]

\let\Section=\section
\def\section{\setcounter{equation}{0}\Section}
\sf

\begin{document}

\title{Multiplicity of normalized solutions for a class of nonlinear Schr\"odinger-Poisson-Slater equations
\thanks{Corresponding author: tingjian.luo@univ-fcomte.fr}}
 \author{Tingjian Luo \\
\\Laboratoire de Math\'ematiques\\UMR CNRS 6623\\Universit\'{e} de Franche-Comt\'{e}\\16 Route de Gray\\25030 Besan\c{c}on Cedex, France}
\date{}
\maketitle

\noindent {\bf Abstract:}  In this paper, we prove a multiplicity result of solutions for the following stationary Schr\"odinger-Poisson-Slater equations
\begin{equation}\label{eq-abstract}
-\Delta u - \lambda u + (\left | x \right |^{-1}\ast \left | u \right |^2) u - |u|^{p-2}u = 0 \ \mbox{ in } \  \mathbb{R}^{3},
\end{equation}
where $\lambda\in \R$ is a parameter, and $p\in (2,6)$. The solutions we obtained have a prescribed $L^2$-norm. Our proofs are mainly inspired by a recent work of Bartsch and De Valeriola \cite{BV}. \\



\section{Introduction}
 We start from the following time-dependent nonlinear Schr\"odinger-Poisson-Slater equations
\begin{equation}\label{eq1}
 i\partial_t \varphi + \Delta \varphi  - (\left | x \right |^{-1}\ast \left | \varphi \right |^2) \varphi + |\varphi|^{p-2}\varphi=0 \  \mbox{ in } \ \mathbb{R} \times \mathbb{R}^{3},
\end{equation}
where $p\in (2,6)$, the unknown $\varphi= \varphi(t,x):\ \R^+\times \R^3 \to \C$ is a complex valued function. This class of Schr\"odinger type equations with a repulsive nonlocal Coulombic potential is obtained by approximation of the  Hartree-Fock equation which has been used to describe a quantum mechanical system of many particles, see for instance \cite{BA, L2, LS, MA}.

Over the past few decades, the equation \eqref{eq1} has been extensively studied. In particular, considerable attention is paid on the study of standing waves for \eqref{eq1}. By standing waves, we mean solutions of \eqref{eq1} with the form
$$\varphi(t,x)= e^{-i \lambda t}u(x),$$
where $\lambda\in \R$ stands for the frequency. Clearly, standing waves $e^{-i \lambda t}u(x)$ solves \eqref{eq1} if and only if the couple $(u, \lambda)$ satisfies the following stationary equation
\begin{equation}\label{E_lambda}
-\Delta u - \lambda u + (\left | x \right |^{-1}\ast \left | u \right |^2) u - |u|^{p-2}u = 0 \ \mbox{ in } \  \mathbb{R}^{3}.
\tag{$E_{\lambda}$}
\end{equation}

In \eqref{E_lambda}, when $\lambda\in \R$ appears as a fixed and assigned parameter, results concerning the existence or non-existence of solutions to \eqref{E_lambda}, has been widely established, see \cite{AP0, AP, DM, HK, R, R2, WZ} and the references therein. In particular, previous results on the multiplicity of solutions to \eqref{E_lambda} have also been obtained. We refer to \cite{ADP2, AR, CKW, GS1} and their references in that direction. In those references, solutions are obtained as critical points of the functional
$$I_{\lambda}(u): = \frac{1}{2}\left \| \triangledown u \right \|_{L^2(\R^3)}^2 - \frac{\lambda}{2}\|u\|_{L^2(\R^3)}^2+\frac{1}{4}\int_{\mathbb{R}^3}\int_{\mathbb{R}^3}\frac{\left | u(x) \right |^2\left | u(y) \right |^2}{\left | x-y \right |}dxdy-\frac{1}{p}\int_{\mathbb{R}^3}\left | u \right |^pdx, $$
which is well-defined and $C^1$ in $H^1(\R^3)$.\medskip

Alternatively, motivated by the fact that people are particularly interested in ``normalized solutions", one can search for solutions of \eqref{E_lambda} having a prescribed norm. Precisely, for given $c>0$, we look to
$$(u_c, \lambda_c)\in H^1(\R^3)\times \R\ \mbox{ solutions of } \eqref{E_lambda} \mbox{ with } \|u\|_{L^2(\R^3)}^2=c.$$

In this case, a solution $u_c\in H^1(\R^3)$ of \eqref{E_lambda} can be obtained as a critical point of the functional
\begin{eqnarray}\label{func1.1}
F(u):= \frac{1}{2}\|\nabla u\|_{L^2(\R^3)}^2 + \frac{1}{4}\int_{\R^3}\int_{\R^3}\frac{|u(x)|^2|u(y)|^2}{|x-y|}dxdy - \frac{1}{p}\int_{\R^3}|u|^pdx,
\end{eqnarray}
on the constraint
\begin{eqnarray}\label{intro-l2-constraint}
S(c): =  \Big \{ u\in H^1(\R^3) : \ \|u\|_{L^2(\R^3)}^2 = c, c>0 \Big\}.
\end{eqnarray}
The parameter $\lambda_c\in \R$, in this situation, is not fixed any more and it appears as a Lagrange parameter.\medskip

The above normalized problem associated to \eqref{E_lambda}, has been studied in the literature \cite{BS1, BS, BJL,CDSS, HLW, JL2, OSJS}. In the cited references, the existence and non-existence of normalized solutions to \eqref{E_lambda} are established, depending strongly on the value of $p\in (2,6)$ and of the parameter $c>0$. Precisely, it is proved that a solution which minimizes globally $F$ on $S(c)$, exists when $p\in (2,3)$ and $c>0$ small enough. When $p\in (3, \frac{10}{3})$, there exists a $c_0>0$ such that such a solution exists if and only if $c\geq c_0$. When $p\in (\frac{10}{3}, 6)$, it is not possible to find a solution as a global minimizer of $F$ on $S(c)$ since $\inf_{u\in S(c)}F(u)=-\infty$, however it is proved in \cite{BJL} that for $c>0$ sufficiently small, $F$ admits a critical point which minimizes the energy among all solutions on $S(c)$.\medskip

Our contribution in this paper is the multiplicity of normalized solutions to \eqref{E_lambda}. Namely, we prove that there exist infinitely many normalized solutions of \eqref{E_lambda}. Up to our knowledge, in the existing literature, results in that direction do not exist yet. The solutions are obtained as critical points of the functional $F$ on the constraint $S(c)$. Our main result is the following theorem:
\begin{theorem}\label{th-multiplicity}
Assume that $p \in (\frac{10}{3},6)$. There exists $c_0>0$ such that for any $c \in (0,c_0)$, the equation \eqref{E_lambda} admits an unbounded sequence of distinct pairs of radial solutions $(\pm u_n, \lambda_n)$ with $\|u_n\|_{L^2(\R^3)}^2 =c$ and $\lambda_n<0$ for each $n\in \N^+$.
\end{theorem}

Now let us underline the difficulties in showing this theorem. First, since the functional $F$ is unbounded from below on $S(c)$ when $p\in (\frac{10}{3},6)$, the genus of the sublevel set $F^{\alpha}:= \{u \in S(c) : F(u) \leq \alpha\}$ is always infinite. Thus to obtain the existence of infinitely many solutions, classical arguments based on the Kranoselski genus, see \cite{MS}, do not apply. \medskip

Secondly, it can be easily checked that the functional $F$, restricted to $S(c)$, does not satisfy the Palais-Smale condition, even working on the subspace $H^1_{r}(\R^3)$ of radially symmetric functions where one has the advantage of the compact embedding of $H_r^1(\R^3)$ into $L^q(\R^3)$ for $q \in (2, 6)$. \medskip


To overcome these difficulties we are inspired by a recent work of Bartsch and De Valeriola \cite{BV}.
In \cite{BV} the authors consider the problem of finding infinitely many critical points for
\begin{eqnarray}
J(u):= \frac{1}{2}\|\nabla u\|_{L^2(\R^3)}^2  - \frac{1}{p}\int_{\R^3}|u|^pdx,
\end{eqnarray}
on the constraint
\begin{eqnarray}\label{c4-constraint}
S_r(c): =  \Big \{ u\in H_r^1(\R^3) : \ \|u\|_{L^2(\R^3)}^2 = c, c>0 \Big\},
\end{eqnarray}
when $p \in (\frac{10}{3}, 6)$.
Actually in \cite{BV} more general nonlinearities can be handled and in any dimension $N \geq 2$. \medskip

In the problem treated in \cite{BV} the difficulties presented above already exist. To overcome these difficulties the authors present a new type of linking geometry for the functional $J$ on $S_r(c)$. This geometry is, according to the authors of \cite{BV}, motivated by the fountain theorem (see \cite{TB1}). In \cite{BV} to set up a min-max scheme and identify a sequence $\{l_n\} \subset \R$, $l_n \to \infty$ of suspected critical levels, the cohomological index for spaces with an action on the group $G=\{-1,1\}$ is used. Indeed observe that the functional $J$ is even, this is also the case of $F$. This index which was introduced in \cite{PCEF} permits to establish the key intersection property, see  \cite[Lemma 2.3]{BV} or our Lemma \ref{lm-paths}. The fact that the suspected critical levels $l_n$ do correspond to critical levels is then obtained using ideas from Jeanjean \cite{LJ}. The key point is the construction, for each fixed $n \in \N$, of a special Palais-Smale sequence associated with $l_n$. That construction leads easily to get the bounededness and further non-vanishing of the Palais-Smale sequence. In that aim one introduces an auxiliary functional which permits to incorporate into the variational procedure the information that any critical point of $J$ on $S_r(c)$ must satisfy a version of Pohozaev identity. Having obtained the boundedness and non-vanishing of the Palais-Smale sequence, the remaining is to show its compactness. The information that the suspected associated Lagrange multiplier is strictly negative is here crucially used. \medskip

In our proof of Theorem \ref{th-multiplicity} we follow mainly the strategy of \cite{BV}. However, due to the nonlocal term $(|x|^{-1}\ast |u|^2)u$ in \eqref{E_lambda}, new treatments are needed for our problem. In particular, the construction of a special Palais-Smale sequence at each suspected critical level is more delicate, see Lemma \ref{lm-PSC}. In addition, the restriction that $c \in (0, c_0)$ originates in the need to show that the suspected associated Lagrange multipliers are strictly negative. This property is used to show that the weak limit of our Palais-Smale sequences does belong to $S_r(c)$. A similar limitation on $c>0$ was already necessary in \cite{BJL} where the existence of just one critical point of $F$ on $S(c)$ is proved. More generally in the proofs of this paper we make use of some results derived in \cite{BJL}.\\

\textbf{Notations}:
In the paper, for simplicity we write $L^{s}(\R^{3}), H_r^{1}(\R^{3}) ....$,  and for any $1\leq s < +\infty$, $L^s(\R^3)$ is the usual Lebesgue space endowed
with the norm
$$\|u\|_{s}^s:=\int_{\R^{3}} |u|^sdx,$$
and $H_r^1(\R^3)$ is the subspace of radially symmetric functions, endowed with the norm
$$\|u\|^2:=\int_{\R^{3}} |\nabla u|^2dx+\int_{\R^{3}} |u|^2dx.$$
Moreover we define, for short,
the following quantities
$$A(u):=\int_{\R^{3}} |\nabla u|^{2}dx,\ \ \ B(u):=\int_{\mathbb{R}^3}\int_{\mathbb{R}^3}\frac{\left | u(x) \right |^2\left | u(y) \right |^2}{\left | x-y \right |}dxdy$$
$$ \ C(u):=\int_{\R^{3}} |u|^{p}dx, \ \ \ D(u):=\int_{\mathbb{R}^3}\left | u\right |^2dx. $$

\section{Proofs of the main results}
We first establish some preliminary results.  Let $\{V_n\}\subset H_r^1(\R^3)$ be a strictly increasing sequence of finite-dimensional linear subspaces in $H_r^1(\R^3)$, such that $\bigcup_n V_n$ is dense in $H_r^1(\R^3)$. We denote by $V_{n}^{\perp}$ the orthogonal space of $V_n$ in $H_r^1(\R^3)$. Then
\begin{lemma}\cite[Lemma 2.1]{BV}\label{lm-mu}
Assume that $p\in (2,6)$. Then there holds
\begin{eqnarray*}
\mu_n:= \inf_{u\in V_{n-1}^{\perp}}\frac{\int_{\R^3}(|\nabla u|^2 + |u|^2)dx}{(\int_{\R^3}|u|^pdx)^{2/p}}= \inf_{u\in V_{n-1}^{\perp}}\frac{\|u\|^2}{\|u\|_p^2} \to \infty,\ \mbox{ as } n\to \infty.
\end{eqnarray*}
\end{lemma}

Now for $c>0$ fixed and for each $n\in \N$, we define
$$\rho_n:= L^{-\frac{2}{p-2}}\cdot \mu_n^{\frac{2}{p-2}},\ \mbox{ with } L= \max_{x>0} \frac{(x^2 + c)^{p/2}}{x^p+ c^{p/2}},$$
and
\begin{eqnarray}\label{B_n}
B_n:= \{u\in V_{n-1}^{\perp}\cap S_r(c):\ \|\nabla u\|_2^2 = \rho_n \}.
\end{eqnarray}
We also define
\begin{eqnarray}\label{bn}
b_n:= \inf_{u\in B_n}F(u).
\end{eqnarray}
Then we have
\begin{lemma}\label{lm-bn}
For any $p\in (2,6)$, $b_n \to +\infty$ as $n\to \infty$. In particular we can assume without restriction that $b_n \geq 1$ for all $n \in \N$.
\end{lemma}
\begin{proof}
For any $u\in B_n$, we have that
\begin{eqnarray*}
F(u) &=& \frac{1}{2}\|\nabla u\|_2^2 + \frac{1}{4}\int_{\R^3}\int_{\R^3}\frac{|u(x)|^2|u(y)|^2}{|x-y|}dxdy - \frac{1}{p}\int_{\R^3}|u|^pdx\\
&\geq& \frac{1}{2}\|\nabla u\|_2^2 - \frac{1}{p\mu_n} \Big(\|\nabla u\|_2^2 + c \Big)^{\frac{p}{2}}\\
&\geq& \frac{1}{2}\|\nabla u\|_2^2 - \frac{L}{p\mu_n} \Big(\|\nabla u\|_2^p + c^{\frac{p}{2}} \Big)\\
&\geq& (\frac{1}{2}- \frac{1}{p})\rho_n - \frac{L}{p\mu_n}c^{\frac{p}{2}}.
\end{eqnarray*}
From this estimate and Lemma \ref{lm-mu}, it follows since  $p>2$, that $b_n\to +\infty$ as $n\to \infty$. Now, considering the sequence $\{V_n\} \subset H_r^1(\R^3)$  only from a $n_0 \in \N$ such that $b_n \geq 1$ for any $n \geq n_0$ it concludes the proof of the lemma.
\end{proof}

Next we start to set up our min-max scheme. First we introduce the map
\begin{eqnarray}\label{kappa}
\kappa: H^1_r(\R^3)\times\R &\longrightarrow& H^1_r(\R^3) \nonumber \\
(u,\ \theta) &\longmapsto& \kappa(u,\theta)(x):= e^{\frac{3}{2}\theta}u(e^{\theta} x).
\end{eqnarray}
Observe that for any given $u\in S_r(c)$, we have $\kappa(u,\theta)\in S_r(c)$ for all $\theta\in \R$. Also from \cite[Lemma 2.1]{BJL}, we know that
\begin{equation}\label{behaviscaling}
\left\{
\begin{array}{l}
A(\kappa(u,\theta))\to 0,\quad F(\kappa(u,\theta))\to 0,\quad \ \mbox{ as } \theta\to -\infty, \\
A(\kappa(u,\theta))\to +\infty,\ F(\kappa(u,\theta))\to -\infty,\ \mbox{ as } \theta\to +\infty.
\end{array}
\right.
\end{equation}
Thus, using the fact that $V_n$ is finite dimensional, we deduce that, for each $n \in \N$, there exists a $\theta_n>0$, such that
\begin{eqnarray}\label{g_n}
\bar{g}_n:\ [0,1]\times (S_r(c)\cap V_n) \to S_r(c), \ \  \bar{g}_n(t, u) = \kappa (u, (2t-1)\theta_n)
\end{eqnarray}
satisfies
\begin{equation}\label{pro-g_n}
\left\{
\begin{array}{l}
A(\bar{g}_n(0,u))< \rho_n,\ A(\bar{g}_n(1,u))> \rho_n, \\
F(\bar{g}_n(0,u)) < b_n ,\ F(\bar{g}_n(1,u)) < b_n.
\end{array}
\right.
\end{equation}
\medskip

Now we define
\begin{align}\label{declass}
\Gamma_n: =\Big\{g:\ &[0,1]\times (S_r(c)\cap V_n) \to S_r(c)\ |\ g \mbox{ is continuous, odd in } u \\
&\mbox{and such that } \forall u:\ g(0,u)= \bar{g}_n(0,u),\ g(1,u)=\bar{g}_n(1,u)\Big\}.
\end{align}
Clearly $\bar{g}_n \in \Gamma_n$. Now we give the key intersection result, due to \cite{BV}.
\begin{lemma}\label{lm-paths}
For each $n\in \N$,
\begin{eqnarray}\label{gamma_n}
\gamma_n(c) := \inf_{g \in \Gamma_n}\max\limits_{\substack{0\leq t\leq 1 \\ u\in S_r(c)\cap V_n}}F(g(t,u))\geq b_n.
\end{eqnarray}
\end{lemma}
\begin{proof}
The point to show that for each $g \in \Gamma_n$ there exists a pair $(t,u)\in [0,1]\times (S_r(c)\cap V_n)$, such that $g(t,u)\in B_n$ with $B_n$ defined in \eqref{B_n}.
But this result can be proved exactly as the corresponding result for $J$ in \cite{BV}, see \cite[Lemma 2.3]{BV}.
\end{proof}

\begin{remark}\label{classical}
Note that by Lemma \ref{lm-paths} and \eqref{pro-g_n} we have that for any $g \in \Gamma_n$
\begin{equation}\label{gamma_n}
\gamma_n(c) \geq b_n >  \max \Big\{ \max_{u\in S_r(c)\cap V_n} F(g(0,u)),  \max_{u\in S_r(c)\cap V_n} F(g(1,u))\Big\}.\nonumber
\end{equation}
\end{remark}

Next, we shall prove that the sequence $\{\gamma_n(c)\}$ is indeed a sequence of critical values for $F$ restricted to $S_r(c)$. In this aim, we first show that there exists a bounded Palais-Smale sequence at each level $\gamma_n(c)$. From now on,  we fix an arbitrary $n\in \N$.
\begin{lemma}\label{lm-PSC}
For any fixed $c>0$, there exists a sequence $\{u_k\}\subset S_r(c)$ satisfying
\begin{eqnarray}\label{PSC}
\left\{
\begin{array}{l}
F(u_k)\to \gamma_n(c), \\
F'|_{S_r(c)}(u_k)\to 0, \quad \mbox{ as } k\to \infty,\\
Q(u_k)\to 0,
\end{array}
\right.
\end{eqnarray}
where
\begin{eqnarray}\label{sequence-Q}
Q(u):= A(u) + \frac{1}{4}B(u) - \frac{3(p-2)}{2p}C(u).
\end{eqnarray}
In particular $\{u_k\}\subset S_r(c)$ is bounded.
\end{lemma}

To find such a Palais-Smale sequence, we apply the approach developed by Jeanjean \cite{LJ}, already applied in \cite{BV}. First, we introduce the auxiliary functional
$$\widetilde{F}: S_r(c) \times \R \to \R,\qquad  (u, \theta)\mapsto F(\kappa(u,\theta)),$$
where $\kappa(u,\theta)$ is given in \eqref{kappa}, and the set
\begin{align*}
\widetilde{\Gamma}_n:= \Big\{\widetilde{g}:\ &[0,1]\times (S_r(c)\cap V_n) \to S_r(c) \times \R \ | \ \widetilde{g} \mbox{ is continuous, odd in } u, \\
&\mbox{and such that } \kappa \circ \widetilde{g} \in \Gamma_n   \Big\}.
\end{align*}
Clearly, for any $g\in \Gamma_n$, $\widetilde{g}:=(g, 0) \in \widetilde{\Gamma}_n$.\\

Observe that defining
$$\widetilde{\gamma}_n(c):= \inf_{\widetilde{g}\in \widetilde{\Gamma}_n}\max\limits_{\substack{0\leq t\leq 1 \\ u\in S_r(c)\cap V_n}}\widetilde{F}(\widetilde{g}(t,u)),$$
we have that $\widetilde{\gamma}_n(c) = \gamma_n(c)$. Indeed, by the definitions of $\widetilde{\gamma}_n(c)$ and $\gamma_n(c)$, this identity follows immediately from the fact that the maps
$$\varphi: \Gamma_n \longrightarrow \widetilde{\Gamma}_n,\ g \longmapsto \varphi(g):=(g,0),$$
and
$$\psi: \widetilde{\Gamma}_n \longrightarrow \Gamma_n,\ \widetilde{g} \longmapsto \psi(\widetilde{g}):=\kappa \circ \widetilde{g},$$
satisfy
$$\widetilde{F}(\varphi(g)) = F(g)\ \mbox{ and }\ F(\kappa \circ \widetilde{g}) = \widetilde{F}(\widetilde{g}).$$

To prove Lemma \ref{lm-PSC} we also need the following result, which was established by Ekeland's variational principle in \cite[Lemma 2.3]{LJ}. We denote by $E$ the set $ H^1_r(\R^3) \times \R$ equipped with $\|\cdot\|_E^2 = \|\cdot\|^2 + |\cdot|_{\R}^2$, and by $E^{\ast}$ its dual space.
\begin{lemma}\label{lm-ekeland1}
Let $\varepsilon>0$. Suppose that $\widetilde{g}_0 \in \widetilde{\Gamma}_n$ satisfies
$$\max\limits_{\substack{0\leq t\leq 1 \\ u\in S_r(c)\cap V_n}}\widetilde{F}(\widetilde{g}_0(t,u))\leq \widetilde{\gamma }_n(c)+ \varepsilon.$$
Then there exists a pair of $(u_0, \theta_0)\in S_r(c)\times \R$ such that:
\begin{itemize}
  \item [(1)] $\widetilde{F}(u_0,\theta_0) \in [\widetilde{\gamma}_n(c)- \varepsilon, \widetilde{\gamma}_n(c) + \varepsilon]$;
  \item [(2)] $\min\limits_{\substack{0\leq t\leq 1 \\ u\in S_r(c)\cap V_n}} \| (u_0,\theta_0)-\widetilde{g}_k(t, u) \|_E \leq \sqrt{\varepsilon}$;
  \item [(3)] $\| \widetilde{F}'|_{S_r(c) \times \R}(u_0,\theta_0)  \|_{E^{\ast}}\leq 2\sqrt{\varepsilon}$,\ i.e.
$$|\langle \widetilde{F}'(u_0, \theta_0), z \rangle_{E^{\ast}\times E}  |\leq 2\sqrt{\varepsilon}\left \| z \right \|_E,$$ holds for all $z \in \widetilde{T}_{(u_0,\theta_0)}:= \{(z_1,z_2) \in E, \langle u_0, z_1\rangle_{L^2}=0\}$.
\end{itemize}
\end{lemma}

Now we can give
\begin{proof}[Proof of Lemma \ref{lm-PSC}]
From the definition of $\gamma_n(c)$, we know that for each $k\in \N$, there exists a $g_k\in \Gamma_n$ such that
$$\max\limits_{\substack{0\leq t\leq 1 \\ u\in S_r(c)\cap V_n}}F(g_k(t, u))\leq \gamma_n(c) + \frac{1}{k}.$$
Since $\widetilde{\gamma}_n(c)=\gamma_n(c)$, $\widetilde{g}_k=(g_k,0)\in \widetilde{\Gamma}_n$ satisfies
$$\max\limits_{\substack{0\leq t\leq 1 \\ u\in S_r(c)\cap V_n}}\widetilde{F}(\widetilde{g}_k(t, u))\leq \widetilde{\gamma}_n(c) + \frac{1}{k}.$$
Thus applying Lemma \ref{lm-ekeland1}, we obtain a sequence $\{(u_k,\theta_k)\}\subset S_r(c) \times \R$ such that:
\begin{itemize}
  \item [(i)] $\widetilde{F}(u_k, \theta_k) \in [\gamma_n(c)-\frac{1}{k}, \gamma_n(c)+\frac{1}{k}]$;
  \item [(ii)] $\min\limits_{\substack{0\leq t\leq 1 \\ u\in S_r(c)\cap V_n}} \| (u_k, \theta_k)-(g_k(t, u), 0) \|_E \leq\frac{1}{\sqrt{k}}$;
  \item [(iii)] $\| \widetilde{F}'|_{S_r(c) \times \R}(u_k, \theta_k)  \|_{E^{\ast}}\leq \frac{2}{\sqrt{k}}$,\ i.e.
$$|\langle \widetilde{F}'(u_k, \theta_k), z \rangle_{E^{\ast}\times E}  |\leq \frac{2}{\sqrt{k}}\left \| z \right \|_E,$$ holds for all $z \in \widetilde{T}_{(u_k, \theta_k)}:= \{(z_1,z_2) \in E, \langle u_k, z_1\rangle_{L^2}=0\}$.
\end{itemize}
For each $k\in \N$, let $v_k=\kappa(u_k, \theta_k)$. We shall prove that $v_k \in S_r(c)$ satisfies \eqref{PSC}. Indeed, first, from $(i)$ we have that $F(v_k) \underset{k}{\to} \gamma_n(c)$, since $F(v_k)=F(\kappa(u_k, \theta_k))=\widetilde{F}(u_k, \theta_k)$. Secondly, note that
$$Q(v_k) = A(v_k) + \frac{1}{4}B(v_k) - \frac{3(p-2)}{2p}C(v_k)= \langle \widetilde{F}'(u_k, \theta_k), (0,1) \rangle_{E^{\ast}\times E},$$
and $(0,1)\in \widetilde{T}_{(u_k, \theta_k)}$. Thus $(iii)$ yields $Q(v_k)\underset{k}{\to} 0$. Finally, to verify that $F'|_{S_r(c)}(v_k)\underset{k}{\to}0$, it suffices to prove for $k\in \N$ sufficiently large, that
\begin{eqnarray}\label{c4-F'u}
|\langle F'(v_k), w\rangle_{(H_r^{1})^{\ast}\times H_r^{1}}  |\leq \frac{4}{\sqrt{k}}\left \| w \right \|^2,\ \mbox{ for all }\ w \in T_{v_k},
\end{eqnarray}
where $T_{v_k}:=\{w \in H_r^{1}(\R^3),\  \langle v_k,w\rangle_{L^2}=0\}$. To this end, we note that, for $w\in T_{v_k}$, setting $\widetilde{w}=\kappa(w, -\theta_k)$, one has
\begin{eqnarray*}
&\ &\langle F'(v_k),w \rangle_{(H_r^{1})^{\ast}\times H_r^{1}} \\
&=& \int_{\R^3}\nabla v_k  \nabla w dx + \int_{\R^3}\int_{\R^3}\frac{|v_k(x)|^2 v_k(y) w(y)}{|x-y|} dxdy -\int_{\R^3}|v_k|^{p-2}v_k w dx\\
&=& e^{2\theta_k}\int_{\R^3}\nabla u_k  \nabla \widetilde{w} dx
+ e^{\theta_k} \int_{\R^3}\int_{\R^3}\frac{|u_k(x)|^2 u_k(y) \widetilde{w}(y)}{|x-y|} dxdy\\
&-& e^{\frac{3(p-2)}{2}\theta_k}\int_{\R^3}|u_k|^{p-2}u_k \widetilde{w} dx
\ = \ \langle \widetilde{F}'(u_k, \theta_k), (\widetilde{w}, 0)\rangle_{E^{\ast}\times E}.
\end{eqnarray*}
If $(\widetilde{w},0)\in \widetilde{T}_{(u_k,\theta_k)}$ and $\|(\widetilde{w},0)\|_E^2\leq 2\|w\|^2$ when $k\in \N$ is sufficiently large, then $(iii)$ implies \eqref{c4-F'u}. To verify these conditions, observes that
$(\widetilde{w},0)\in \widetilde{T}_{(u_k,\theta_k))} \Leftrightarrow w \in T_{v_k}$. Also from $(ii)$ it follows that
$$|\theta_k|=|\theta_k-0|\leq \min\limits_{\substack{0\leq t\leq 1 \\ u\in S_r(c)\cap V_n}}\|(v_k, \theta_k)-(g_k(t, u), 0)\|_E\leq \frac{1}{\sqrt{k}},$$
by which we deduce that
\begin{eqnarray*}
\|(\widetilde{w},0)\|_E^2=\|\widetilde{w}\|^2= \int_{\R^3}|w(x)|^2dx + e^{-2\theta_k}\int_{\R^3}|\nabla w(x)|^2dx \leq 2 \|w\|^2,
\end{eqnarray*}
holds for $k\in \N$ large enough. At this point, \eqref{c4-F'u} has been verified. To end the proof of the lemma it remains to show that $\{v_k\} \subset S_r(c)$ is bounded. But since $p\in (\frac{10}{3}, 6)$  this follows immediately from the following relationship between $F(u)$ and $Q(u)$,
\begin{eqnarray}\label{FandQ}
F(u)- \frac{2}{3(p-2)}Q(u) = \frac{3p-10}{6(p-2)}A(u) + \frac{3p-8}{12(p-2)}B(u).
\end{eqnarray}
\end{proof}

\begin{remark}\label{rek-Q}
Note that in \cite[Lemma 2.1]{JL2}, it is proved that any critical point $u_0\in S_r(c)$ of $F$ on $S_r(c)$ must satisfy $Q(u_0)=0$. So far this information has been used in Lemma \ref{lm-PSC} to construct a bounded Palais-Smale sequence. As we shall see in our next result it is also useful to insure that our Palais-Smale sequences do not vanish.
\end{remark}

\begin{proposition}\label{prop-compact}
Let $\{u_k\}\subset S_r(c)$ be the Palais-Smale sequence obtained in Lemma \ref{lm-PSC}. Then there exist $\lambda_n\in \R$ and $u_n\in H_r^1(\R^3)$, such that, up to a subsequence,
\begin{enumerate}
  \item [i)] $u_k \rightharpoonup u_n \neq 0$, in $H^1_r(\R^3)$,
  \item [ii)] $- \Delta u_k - \lambda_n u_k + (|x|^{-1}\ast |u_k|^2)u_k - |u_k|^{p-2}u_k \to 0$,  in $H^{-1}_r(\R^3)$,
  \item [iii)] $- \Delta u_n - \lambda_n u_n + (|x|^{-1}\ast |u_n|^2)u_n - |u_n|^{p-2}u_n = 0$,  in $H^{-1}_r(\R^3)$.
\end{enumerate}
Moreover, if $\lambda_n <0$, then we have
$$u_k \to u_n,\ \mbox{ in } H_r^1(\R^3),\  \mbox{ as } k\to \infty.$$
In particular, $||u_n||_2^2 = c$, $F(u_n)=\gamma_n(c)$ and $F'(u_n)-\lambda_n u_n = 0$ in $H^{-1}_r(\R^3)$.
\end{proposition}

\begin{proof}
Since $\{u_k\}\subset S_r(c)$ is bounded, up to a subsequence, there exists a $u_n\in H^1_r(\R^3)$, such that
$$u_k \underset{k}{\rightharpoonup} u_n,\ \mbox{ in } H^1_r(\R^3),$$
$$u_k \underset{k}{\to} u_n,\ \mbox{ in } L^p(\R^3).$$
We have $u_n \neq 0$. Indeed suppose by contradiction that $u_n =0$. Then by the strong convergence in $L^p(\R^3)$ it follows that $C(u_k) \to 0$. Taking into account that  $Q(u_k)\to 0$ it then implies that  $A(u_k)\to 0$ and $ B(u_k)\to 0$. Thus $F(u_k)\to 0$ and this contradicts the fact that $\gamma_n(c) \geq b_n \geq1$. Thus Point $i)$ holds.\medskip

The proofs of Points $ii)$ and $iii)$ can be found in \cite[Proposition 4.1]{BJL}. Now using Points $ii)$, $iii)$, and the convergence $C(u_k)\underset{k}{\to} C(u_n)$, it follows that
$$A(u_k)-\lambda_n D(u_k) + B(u_k) \underset{k}{\to} A(u_n)-\lambda_n D(u_n) + B(u_n).$$
If $\lambda_n<0$, then we conclude from the weak convergence of $u_k \underset{k}{\rightharpoonup} u_n$ in $H^1_r(\R^3)$, that
$$A(u_k)\underset{k}{\to}A(u_n),\,  B(u_k)\underset{k}{\to}B(u_n), \,  C(u_k)\underset{k}{\to}C(u_n).$$
Thus $u_k\underset{k}{\to}u_n$ in $H^1_r(\R^3)$, and in particular, $||u_n||_2^2 =c$, $F(u_n)=\gamma_n(c)$ and $F'(u_n)-\lambda_n u_n = 0$ in $H^{-1}_r(\R^3)$.
\end{proof}

At this point we can prove our main result.



\begin{proof}[Proof of Theorem \ref{th-multiplicity}]
By Lemma \ref{lm-PSC} and Proposition \ref{prop-compact}, to prove Theorem \ref{th-multiplicity}, it is enough to verify that if $(u_n, \lambda_n)\in S_r(c)\times \R$ solves
$$-\Delta u - \lambda u + (|x|^{-1}\ast |u|^2)u = |u|^{p-2}u,\ \mbox{ in } \R^3,$$
then necessarily $\lambda_n<0$ provided $c>0$ is sufficiently small. However, this point has been proved in \cite[Lemma 4.2]{BJL}. Thus the proof of the theorem is completed.
\end{proof}

\textbf{Acknowledgements.} The author would like to thank Professor Louis Jeanjean for helpful discussions  and useful comments on a preliminary version of this paper.

\end{document}